\title{Clique-Relaxed Competitive Graph Coloring}
\author{
    Michel Alexis, Davis Shurbert, \\
    Dr.\ Charles Dunn, Dr.\ Jennifer Nordstrom
           }
    
\documentclass[12pt]{article}

\usepackage{graphicx}
\usepackage{tikz}
\usetikzlibrary{matrix}

\usepackage{arrayjob}
\usepackage{setspace}
\usepackage{multicol}
\usepackage{relsize}
\usepackage{verbatim}
\usepackage{amssymb}
\usepackage{lipsum}
\usepackage{color}
\usepackage{hyperref}
\usepackage{amsmath}
\usepackage{amsthm}
\usepackage{scrextend}
\usepackage{listings}
\usepackage{graphicx}
\usepackage{pdfpages}

\usepackage[margin=1in]{geometry}

\newcommand{\adj}[0]{
    \leftrightarrow
}

\newcommand{\maj}[0]{
    \textnormal{maj}
} \newcommand{\coloneqq}[0]{
    :=
}

\newcommand*\makeAlph[1]{\symbol{\numexpr96+#1}}

\theoremstyle{definition}
\newtheorem{definition}{Definition}[section]
\newtheorem{theorem}[definition]{Theorem}
\newtheorem{lemma}[definition]{Lemma}
\newtheorem{claim}{Claim}[definition]
\newtheorem{corollary}[definition]{Corollary}

\begin{document}
\maketitle

\section{Introduction}

In this paper we investigate a variation of the graph coloring game, as studied
in \cite{clique-relaxed}. In the original coloring game, two players, Alice and
Bob, alternate coloring vertices on a graph with legal colors from a fixed color
set, where a color $\alpha$ is \emph{legal} for a vertex if said vertex has no
neighbors colored $\alpha$. Other variations of the game change this definition
of a legal color. For a fixed color set, Alice wins the game if all vertices are
colored when the game ends, while Bob wins if there is a point in the game in
which a vertex cannot be assigned a legal color. The least number of colors
needed for Alice to have a winning strategy on a graph $G$ is called the
\emph{game chromatic number of G}, and is denoted $\chi_{\text{g}} (G)$. A well
studied variation is the \emph{$d$-relaxed coloring game} \cite{defective} in
which a legal coloring of a graph $G$ is defined as any assignment of colors to
$V(G)$ such that the subgraph of $G$ induced by any color class has maximum
degree $d$.

We focus on the \emph{$k$-clique-relaxed $n$-coloring game}. A
\emph{$k$-clique-relaxed $n$-coloring} of a graph $G$ is an $n$-coloring in
which the subgraph of $G$ induced by any color class has maximum clique size $k$
or less. In other words, a $k$-clique-relaxed $n$-coloring of $G$ is an
assignment of $n$ colors to $V(G)$ in which there are no monochromatic
$(k+1)$-cliques.

For a fixed color set, the $k$-clique-relaxed coloring game begins with Alice
coloring any vertex. Alice and Bob then alternate coloring vertices of $G$ such
that at no point in the game does there exists a monochromatic $(k+1)$-clique.
Again, Alice wins the game if it terminates with all vertices colored, while Bob
wins otherwise. The smallest number of colors needed for Alice to have a winning
strategy on $G$ is called the \emph{$k$-clique-relaxed game chromatic number of
$G$}, and is denoted $\chi^{(k)}_{\text{g}} (G)$. We use a modification of a
well known strategy for the original coloring game in order to construct a
winning strategy for Alice. Furthermore, we focus on the game as it is played on
chordal graphs and partial $k$-trees. We extend results from \cite{forests} and
\cite{clique-relaxed} by specifically looking at the $k$-clique-relaxed coloring
game as played on partial $k$-trees.

\section{Background}
Before exploring how the game plays out on particular graphs, we must go over
some basic definitions as well as the strategy used by Alice.

\subsection{Terminology}

Let $G$ be a graph with linear ordering $\mathcal{L} = v_1 , \ldots , v_n$ on
the vertices. When a vertex $x$ has a lower index in $\mathcal{L}$ than a vertex
$y$, we denote this relationship by \emph{$x<y$}.


\begin{definition}

For any vertex $v$ of $G$, we define the \emph{back-neighbors} or \emph{parents}
of $v$ as the neighbors of $v$ that are strictly less than $v$ in $\mathcal{L}$.
Similarly, we define the \emph{forward-neighbors} or \emph{children} of $v$ as
the neighbors of $v$ that are strictly greater than $v$ in $\mathcal{L}$. The
set of back-neighbors and forward-neighbors of a vertex $v$ are denoted $N^{+}
(v) = { \lbrace  x ~|~ x \leftrightarrow v \text{ and } x < v } \rbrace$ and
$N^{-} (v) = { \lbrace  x ~|~ x \leftrightarrow v \text{ and } x > v } \rbrace$
respectively. In addition, we define the sets $N^+[v] = N^+(v) \cup \{ v \}$ and
$N^-[v] = N^-(v) \cup \{ v \}$. In order to further emphasize the importance of
the linear ordering, we will visually represent an ordering $\mathcal{L}$ as an
orientation of $G$ such that for two vertices $x$ and $y$, $x$ is an outneighbor
of $y$ if and only if $x<y$ and $x \adj y$.

\end{definition}

\begin{definition}

We define the \emph{major parent} of a vertex $v$ with respect to $\mathcal{L}$
as the least element of $N^{+} (v)$. We denote the major parent of $v$ by
$\maj(v)$. Note that $\maj (v) = \min N^{+} (v)$ under $\mathcal{L}$.

\end{definition}

\begin{definition}

At any given point in the game, we define the following dynamic sets.

\begin{itemize}
\item The set of uncolored vertices $U = { \lbrace v \in V(G) ~|~ v \text{ is not colored}} \rbrace$.
\item The set of colored vertices $C = { \lbrace v \in V(G) ~|~ v \text { is colored}} \rbrace$.
\item A set of active vertices $A$ to be maintained by Alice. Using her
strategy, Alice will always activate a vertex before she colors it. When Bob
colors a vertex $v$, Alice immediately activates $v$. Note that $C \subseteq A$.
\end{itemize}
\end{definition}



\begin{definition}

For any $x \in C$, let $c(x)$ denote the color of $x$.

\end{definition}

\begin{definition}

We define a vertex $w$ to be the \emph{mother} of $x$ if $w$ is the least
uncolored element of $N^+ [x]$ with respect to $\mathcal{L}$. We denote this
vertex $m(x)$.

\end{definition}

Note that if $x$ and all of its parents are already colored, $m(x)$ does not
exist. Furthermore, if $x \in U$, then $x$ might be its own mother. Also note
that the mother of a vertex $x$ will change throughout the course of the game.

%
%

%
%
%
%

\subsection{The Activation Strategy}

We will now outline the \emph{Activation Strategy} to be used by Alice, a
variation of the strategy used in \cite{chordal}. Given a graph $G$, Alice
chooses a linear ordering $\mathcal{L}$ (we will outline later what specific
properties we want $\mathcal{L}$ to satisfy) which she will reference throughout
the game. On her first turn, Alice activates and colors the least vertex in
$\mathcal{L}$. When Bob colors a vertex $b$, Alice's move will be conducted in
two stages: a search stage, in which she searches for a vertex to color, and a
coloring stage, in which she decides which color to use on the vertex.

\paragraph{Search Stage}
This stage will be carried out in two steps: an initial step, and a recursive
step (which may be skipped depending on the initial step).

\subparagraph{Initial Step}

Once Bob has colored vertex $b$, Alice activates $b$ and proceeds to search for
the mother $m(b)$. If $m(b)$ exists, Alice sets $x \coloneqq m(b)$ and moves to
the recursive step. If $m(b)$ does not exist, then Alice selects the least
uncolored vertex $u$ in $\mathcal{L}$, and proceeds to the coloring stage.

\subparagraph{Recursive Step}

In this step, Alice looks at some uncolored vertex $x$. If $x$ is inactive,
Alice activates $x$, sets $x \coloneqq m(x)$ and repeats the recursive step. If
$x$ is active, Alice sets $u \coloneqq x$ and proceeds to the coloring stage.

\paragraph{Coloring Stage}

Alice chooses a legal color for $u$.

\section{Chordal Graphs and Partial $k$-Trees}

In this section we analyze the performance of Alice's Activation Strategy for
the clique-relaxed coloring game on chordal graphs and partial $k$-trees. The
benefit of studying these classes of graphs is that partial $k$-trees encompass
a multitude of graphs, while chordal graphs maintain specific properties which
prove to be beneficial while playing the clique-relaxed coloring game.

Recall that a graph $G$ is \emph{chordal} if $G$ does not contain $C_n$ with $n
\geq 4$ as an induced subgraph. Furthermore, recall that a \emph{$k$-clique} is
defined as a complete subgraph of a graph $H$ over $k$ vertices. The
\emph{clique number} of $G$, denoted $\omega(G)$, is the size of its largest
clique.

\begin{definition}

A \emph{partial $k$-tree} is a graph $G$ which is a subgraph of a chordal graph
$H$, where $\omega(H) = k+1$ and $V(G)=V(H)$.

\end{definition}

\begin{theorem}
\label{simplicial}
\cite{West} $H$ is a chordal graph if and only if there exists a linear ordering
$\mathcal{L}$ of the vertices of $H$ such that for each vertex $v \in V(H)$, the
subgraph induced by the elements of $N^{+} [v]$ with respect to $\mathcal{L}$
form a clique. We call such an ordering a \emph{simplicial ordering} on $H$.
\end{theorem}
%
%
%
%
\begin{lemma}
\label{backneighbors}
Let $H$ be a chordal graph with clique number $\omega (H) = k+1$ and simplicial
ordering $\mathcal{L} = v_1 v_2 \ldots v_n $. Then for any vertex $x \in V(H)$,
$x$ has at most $k$ parents in $\mathcal{L}$.
\end{lemma}
\begin{proof}
If $\mathcal{L}$ is a simplicial ordering for $H$, then for each $x \in  V(H)$
the elements of $N^+[x]$ must form a clique. Since $\omega (H) = k+1$, $N^+[x]$
can have at most $k+1$ elements. Hence the set of parents of $x$, $N^+(x) =
N^+[x] \setminus \lbrace { x \rbrace }$ can have at most $k$ elements.
\end{proof}

\begin{theorem}
\label{k+3}
Let $H$ be a chordal graph with clique number $\omega(H) = k+1$. Then
$\chi_{\text{g}}^{(k)} (H) \leq k+3$.
\end{theorem}
\begin{proof}
Assume Alice and Bob are playing the $k$-clique-relaxed coloring game with $k+3$
colors $\alpha_1, \ldots, \alpha_{k+3}$ on $H$. Alice will play using the
Activation Strategy. It suffices to show that for any uncolored vertex $v \in
V(H)$, the strategy provides a legal coloring for Alice. This is sufficient to
show that both players always have a valid move, as a legal coloring for Alice
is also a legal coloring for Bob. Let Alice pick a simplicial linear ordering
$\mathcal{L}$ of the vertices of H to use for the activation strategy.

Let $g$ be an uncolored vertex in $V(H)$. The only way for a vertex $g$ to have
no legal color is if $g$ belongs to $k+3$ $(k+1)$-cliques whose only pairwise
common vertex is $g$. Label these cliques $\mathcal{C}_1, \ldots,
\mathcal{C}_{k+3}$, where for each $i \in [k+3]$ all vertices in $\mathcal{C}_i$
except for $g$ are colored $\alpha_i$. We will now show that such a situation
can never arise.

By Lemma \ref{backneighbors}, since $\omega(H) = k+1$, $g$ can have at most $k$
back-neighbors. Since we have $k+3$ different cliques, there must be at least
three cliques $\mathcal{C}_{\alpha}, \mathcal{C}_{\beta}$, and
$\mathcal{C}_{\gamma}$ which contain no back-neighbors of $g$. Let $a,b,$ and
$c$ be the vertices of maximal index in $\mathcal{L}$ in $\mathcal{C}_{\alpha},
\mathcal{C}_{\beta}$, and $\mathcal{C}_{\gamma}$ respectively. Then $a,b$, and
$c$ must be forward-neighbors of all the vertices in their respective cliques.
However, any vertex in $H$ can have at most $k$ back-neighbors with respect to
$\mathcal{L}$. As a result, $a, b$, and $c$ can have no back-neighbors outside
of $\mathcal{C}_{\alpha}, \mathcal{C}_{\beta}$, and $\mathcal{C}_{\gamma}$
respectively. We may now conclude that $V(\mathcal{C}_{\alpha}) = N^+[a]$,
$V(\mathcal{C}_{\beta}) = N^+[b]$, and $V(\mathcal{C}_{\gamma}) = N^+[c]$.
Furthermore, since $g$ has no back-neighbors in $\mathcal{C}_{\alpha}$,
$\mathcal{C}_{\beta}$, or $\mathcal{C}_{\gamma}$, $g$ is the vertex of least
index in all three of these cliques. Hence when the first two of $a$, $b$, and
$c$ are activated, if $g$ is uncolored then Alice will take action on $g$ as
$m(a)=m(b)=m(c)=g$. As at most two actions can be taken on $g$, $g$ will be
colored before $a$, $b$, and $c$ can all be colored. Hence the situation where
$g$ is uncolorable will never arise and Alice can always win using the
Activation Strategy with $k+3$ colors.
\end{proof}

\begin{corollary}
\label{k+3corollary}
Let $G$ be a partial $k$-tree. Then the $k$-clique-relaxed game chromatic number
of $G$ is less than or equal to $k+3$. That is, $\chi_\text{g}^{(k)} (G) \leq
k+3$.
\end{corollary}

\begin{proof}

Assume Alice and Bob are playing the coloring game with $k+3$ colors $\alpha_1,
\ldots, \alpha_{k+3}$ on $G$. Since $G$ is a partial $k$-tree, it is the
subgraph of some chordal graph $H$  with clique number $\omega(H) \leq k+1$ and
with the same vertex set as $G$. Alice will play the $k$-clique-relaxed coloring
game on $G$ as though she were playing the game on $H$ using the activation
strategy. Hence any of Alice's moves on $H$ will still be valid in $G$, as the
extra edges in $H$ only decrease the number of legal moves available to her in
$G$. It suffices to show that for any uncolored vertex $v \in V(H)$, the
strategy provides a legal coloring for Alice, which is shown in Theorem
\ref{k+3}.
\end{proof}

Recall that every forest is a partial $1$-tree, and note that playing the
$1$-clique-relaxed coloring game is the same as playing the original coloring
game on a graph. The above result is a generalization of a result in
\cite{forests} where the authors show that at most $4$ colors are needed for
Alice to have a winning strategy on any given forest. Next we generalize a
result from a previous REU \cite{clique-relaxed} in which the authors found that
the $2$-clique-relaxed game chromatic number of an outerplanar graph is at most
$4$.

\begin{theorem}
\label{chordalfourbound}
If $H$ is a chordal graph with clique number $\omega(H) = 3$, then
$\chi_{\text{g}}^{(2)} (H) \leq 4$.
\end{theorem}
\begin{proof}
While playing the 2-clique-relaxed game on such a graph $H$, the only scenario
in which an uncolored vertex $g$ cannot be colored is when $g$ belongs to 4
otherwise disjoint 3-cliques where the vertices of each clique other than $g$
are monochromatically colored such that all four colors are used. Such a
situation can be seen in Figure \ref{fig:figure1}, where $c(a)=c(b)=1$,
$c(c)=c(d)=2$, $c(e)=c(f)=3$, and $c(h)=c(i)=4$. We define $N_g =
\{a,b,...,g,...,i \}$.

\begin{figure}[here]
\caption{}
\label{fig:figure1}
\begin{center}
\begin{tikzpicture}[auto, thick,
  main node/.style={circle,fill=white!10,draw,font=\sffamily\tiny\bfseries},
  colored node/.style={circle,fill=gray!10,draw,font=\sffamily\tiny\bfseries},
  ]

 \def \n {6}
 \def \radius {1.5cm}
 \def \margin {8} 
  \node[main node] (G) at (0,0) {g};
  \foreach \s in {1,...,\n}
  {
   \node[colored node] (\s) at ({360/8 * (\s - 1) + 360/4 + 360/16}:\radius) {\makeAlph{\s}};
    \path
   (G) edge node {} (\s);
  }
  \node[colored node] (7) at ({360/8 * (6) + 360/4 + 360/16}:\radius) {h};
  \node[colored node] (8) at ({360/8 * (7) + 360/4 + 360/16}:\radius) {i};
  
  \node[above] at (1.+60) {\tiny{$\alpha_1$}};
  \node[above] at (2.+60) {\tiny{$\alpha_1$}};
  \node[below] at (3.-100) {\tiny{$\alpha_2$}};
  \node[below] at (4.-100) {\tiny{$\alpha_2$}};
  \node[below] at (5.-100) {\tiny{$\alpha_3$}};
  \node[below] at (6.-100) {\tiny{$\alpha_3$}};
  \node[below] at (7.+200) {\tiny{$\alpha_4$}};
  \node[right] at (8.+60) {\tiny{$\alpha_4$}};
  
  \path
  (G) edge node {} (7)
  (G) edge node {} (8)
  (1) edge node {} (2)
  (3) edge node {} (4)
  (5) edge node {} (6)
  (7) edge node {} (8);
  
\end{tikzpicture}
\end{center} 
\end{figure}

Assume that Alice plays using the activation strategy, using a simplicial
ordering $\mathcal{L}$ of the vertices of $H$. By Lemma \ref{backneighbors},
each vertex can have at most $2$ backneighbors. This leaves us with a number of
possibilities for the placement of $g$ in the indexing of ${N_g = \{a,b,\dots,g,
\dots, i \}}$. We assume without loss of generality, $a < b$, $c < d$, $e < f$,
and $h < i$ in this ordering. First, consider the case where $g$ has the least
index in $N_g$. By Lemma \ref{backneighbors}, $g$ must then be the major parent
of all vertices in $\{b, d, f, i\}$ as they cannot have any other parents
outside of $N_g$ without violating the maximum clique size of $g$. Thus,
whenever a vertex in $\{b, d, f, i\}$ is activated, Alice takes action on $g$ so
long as it is uncolored. As a result, $g$ will be colored before all four of
these vertices have been activated, and Alice can legally color $g$.

Now assume that $g$ is not the vertex of least index in $N_g$. Without loss of
generality, assume that $a$ is the vertex of least index in $N_g$. Similar to
the previous case, if $a<g$ are the two vertices of least index in $N_g$, then
$\{d, f, i \}$ all have $g$ as their major parent. Activation on any of these
three vertices will then cause Alice to act on $g$, and Bob cannot cause $g$ to
be uncolorable. Now, if $a<b<g$ are the vertices of least index in $N_g$, then
the activation of any vertex in $\{d, f, i \}$ will cause Alice to take action
on $g$. This again ensures that Alice can always properly color $g$.

Now, our remaining case is one in which the two back-neighbors allowed to $g$
both belong to different cliques. Assume without loss of generality that $a<c<g$
are the vertices of least index in $N_g$ with respect to our simplicial linear
ordering, as depicted in Figure \ref{fig:figure2}. Note that since $N^+[g]$ must
form a clique, then $a \adj c$ and $a$ is a parent of $c$. We now consider a
number of Claims.

\begin{figure}[here]
\caption{$a<c<g$}
\label{fig:figure2}
\begin{center}
\begin{tikzpicture}[auto, thick,
  main node/.style={circle,fill=white!10,draw,font=\sffamily\tiny\bfseries},
  colored node/.style={circle,fill=gray!10,draw,font=\sffamily\tiny\bfseries},
  ]

 \def \n {6}
 \def \radius {1.5cm}
 \def \margin {8} 
  \node[main node] (G) at (0,0) {g};
  \foreach \s in {4,...,\n}
  {
   \node[main node] (\s) at ({360/8 * (\s - 1) + 360/2 + 360/16}:\radius) {\makeAlph{\s}};
    \path
   (\s) edge [->] node {} (G);
  }
  \node[main node] (1) at ({360/8 * (1) + 360/2 + 360/16}:\radius) {a};
  \node[main node] (2) at ({360/8 * (0) + 360/2 + 360/16}:\radius) {b};
  \node[main node] (3) at ({360/8 * (2) + 360/2 + 360/16}:\radius) {c};
  
  \node[main node] (7) at ({360/8 * (6) + 360/2 + 360/16}:\radius) {h};
  \node[main node] (8) at ({360/8 * (7) + 360/2 + 360/16}:\radius) {i};

  \path
  (7) edge [->] node {} (G)
  (8) edge [->] node {} (G)
  (G) edge [->] node {} (1)
      edge [->] node {} (3)
  (2) edge [->] node {} (G)     
  (2) edge [->] node {} (1)
  (3) edge [->] node {} (1)
  (4) edge [->] node {} (3)
  (6) edge [->] node {} (5)
  (8) edge [->] node {} (7);
  
\end{tikzpicture} 
\vspace{5mm}

\begin{tikzpicture}[auto, thin, scale=1.75,
  main node/.style={circle,inner sep=1.25mm,fill=white!10,draw,font=\sffamily\tiny\bfseries},
  colored node/.style={circle,fill=gray!10,draw,font=\sffamily\tiny\bfseries},
  ]
 \def \n {9}
 \def \radius {1.5cm}
 \def \margin {1} 
  \foreach \s/\na in {1/a,2/c,3/g,4/b,5/d,6/e,7/f,8/h,9/i}
  {
   \node[main node] (\s) at (\intcalcMul{\s}{\margin},0) {\na};
  }
  \path
  (2) edge [->] node {} (1)
  (3) edge [->,bend right] node {} (1)
      edge [->] node {} (2)
  (4) edge [->,bend right] node {} (1)
      edge [->] node {} (3)
  (5) edge [->,bend left] node {} (3)
      edge [->,bend right] node {} (2)
  (6) edge [->,bend left] node {} (3)
  (7) edge [->,bend left] node {} (3)
      edge [->] node {} (6)
  (8) edge [->,bend left] node {} (3)
  (9) edge [->,bend left] node {} (3)
      edge [->] node {} (8);
  
\end{tikzpicture}  
\end{center} 
\end{figure}

\begin{claim}
\label{claim1}
If $a<c<g$, then $g$ is the major parent of $f$ and $i$.
\end{claim}
\begin{proof}
Note that $f > e > g$ and $i > h > g$. By Lemma \ref{backneighbors}, $f$ and
$i$ have at most two back-neighbors. Since $f$ is adjacent to $e$ and $g$, and
$i$ is adjacent to $h$ and $g$, then we must have $N^{+} (f) = \lbrace { g, e
\rbrace }$ and $N^{+} (i) = \lbrace { g, h \rbrace }$. Since $g$ is the minimum
of both sets, we get $\maj(f) = g$ and $\maj(i) = g$.
\end{proof}

\begin{claim}
\label{claim2}
When $a<c<g$, if $a$ is colored before $b$ is activated, or $c$ is colored
before $d$ is activated, $g$ will be legally colored before all vertices in $N_g
\setminus \{ g \}$ are activated.
\end{claim}
\begin{proof}
Assume $a$ is colored before $b$ is activated. When $b$ is finally activated,
its mother $m(b)$ will also be acted upon if $m(b)$ exists. By Lemma
\ref{backneighbors}, $b$ can have at most two back-neighbors. As $b > g > a$
with $b$ adjacent to $g$ and $a$, we can conclude that $g$ and $a$ are the two
back-neighbors of $b$ and $N^{+}(b) = \lbrace { a,g \rbrace }$. The vertex $g$
is uncolored and is thus a candidate for $m(b)$, so we can conclude that $m(b)$
must exist. Since $a$ is already colored, we must have $m(b)=g$. As a result,
when $b$ is activated, action will be taken on $g$ by Alice. Thus, when the
first two active vertices among $b$, $f$, and $i$ are activated, $g$ will have
had two actions taken upon it by Alice. Since only two actions can be taken on a
vertex --- activation and coloring --- $g$ will have been colored. Hence, $g$
will have been colored before all eight surrounding vertices have been
activated. On the other hand, if $c$ is colored before $d$ is activated, then by
a similar argument $g$ will be colored before all vertices in $N_g \setminus \{
g \}$ have been activated.
\end{proof}

Thus by \ref{claim2}, if $a$ is colored before $b$ is activated, or $c$ is
colored before $d$ is activated, then Alice will win. So we will assume from now
onwards that $b$ is activated before $a$ is colored, and $d$ is activated before
$c$ is colored.
Note that $g$ is a candidate for the major parent of both $e$ and $h$. If $g$ is
the major parent of either $e$ or $h$, then we will have three vertices in $N_g
\setminus \{g\}$ which when activated will cause Alice to act on $g$. Assume now
that both $e$ and $h$ have major parents less than $g$. In this case, if $e$ or
$h$ had a major parent $x$ outside of $N_g$, then $x$ would be adjacent to $g$
by Theorem \ref{simplicial}. Hence $g$ would have 3 back-neighbors --- $a$, $c$,
and $x$ --- violating Lemma \ref{backneighbors}. As a result, we can conclude
that $e$ and $h$ both have major parents in $\{a, c\}$. We must now explore two
subcases associated with this new fact. Notice that both subgraphs in Figure
\ref{fig:figure3} are included in the subcase where we assume that $h$ and $e$
share a major parent in $\{a, c\}$, while Figure \ref{fig:figure4} represents
the subcase where $h$ and $e$ do not share a major parent.

\begin{figure}[here]
\caption{$a<c<g$, $h$ and $e$ share a major parent.}
\label{fig:figure3}
\begin{multicols}{2}
\begin{center}
\begin{tikzpicture}[auto, thick,
  main node/.style={circle,fill=white!10,draw,font=\sffamily\tiny\bfseries},
  colored node/.style={circle,fill=gray!10,draw,font=\sffamily\tiny\bfseries},
  ]

 \def \n {6}
 \def \radius {1.5cm}
 \def \margin {8} 
  \node[main node] (G) at (0,0) {g};
  \foreach \s in {4,...,\n}
  {
   \node[main node] (\s) at ({360/8 * (\s - 1) + 360/2 + 360/16}:\radius) {\makeAlph{\s}};
    \path
   (\s) edge [->] node {} (G);
  }
  \node[main node] (1) at ({360/8 * (1) + 360/2 + 360/16}:\radius) {a};
  \node[main node] (2) at ({360/8 * (0) + 360/2 + 360/16}:\radius) {b};
  \node[main node] (3) at ({360/8 * (2) + 360/2 + 360/16}:\radius) {c};
  
  \node[main node] (7) at ({360/8 * (6) + 360/2 + 360/16}:\radius) {h};
  \node[main node] (8) at ({360/8 * (7) + 360/2 + 360/16}:\radius) {i};

  \path
  (7) edge [->] node {} (G)
  (8) edge [->] node {} (G)
  (G) edge [->] node {} (1)
      edge [->] node {} (3)
  (2) edge [->] node {} (G)     
  (2) edge [->] node {} (1)
  (3) edge [->] node {} (1)
  (4) edge [->] node {} (3)
  (6) edge [->] node {} (5)
  (8) edge [->] node {} (7)
  (7) edge [->,dashed] node {} (1)
  (5) edge [->,dashed] node {} (1);
  
\end{tikzpicture}

  \vspace{5mm}

  \begin{tikzpicture}[auto, thin, scale=0.85,
  main node/.style={circle,inner sep=.75mm,fill=white!10,draw,font=\sffamily\tiny\bfseries},
  colored node/.style={circle,fill=gray!10,draw,font=\sffamily\tiny\bfseries},
  ]
 \def \n {9}
 \def \radius {1.5cm}
 \def \margin {1} 
  \foreach \s/\na in {1/a,2/c,3/g,4/b,5/d,6/e,7/f,8/h,9/i}
  {
   \node[main node] (\s) at (\intcalcMul{\s}{\margin},0) {\na};
  }
  \path
  (2) edge [->] node {} (1)
  (3) edge [->,bend right] node {} (1)
      edge [->] node {} (2)
  (4) edge [->,bend right] node {} (1)
      edge [->] node {} (3)
  (5) edge [->,bend left] node {} (3)
      edge [->,bend right] node {} (2)
  (6) edge [->,bend left] node {} (3)
  (7) edge [->,bend left] node {} (3)
      edge [->] node {} (6)
  (8) edge [->,bend left] node {} (3)
  (9) edge [->,bend left] node {} (3)
      edge [->] node {} (8)
  (8) edge [->,bend right,dashed] node {} (1)
  (6) edge [->,bend right,dashed] node {} (1);
  
\end{tikzpicture}

 \end{center}
 \columnbreak
 \begin{center}
 
 \begin{tikzpicture}[auto, thick,
  main node/.style={circle,fill=white!10,draw,font=\sffamily\tiny\bfseries},
  colored node/.style={circle,fill=gray!10,draw,font=\sffamily\tiny\bfseries},
  ]

 \def \n {6}
 \def \radius {1.5cm}
 \def \margin {8} 
  \node[main node] (G) at (0,0) {g};
  \foreach \s in {4,...,\n}
  {
   \node[main node] (\s) at ({360/8 * (\s - 1) + 360/2 + 360/16}:\radius) {\makeAlph{\s}};
    \path
   (\s) edge [->] node {} (G);
  }
  \node[main node] (1) at ({360/8 * (1) + 360/2 + 360/16}:\radius) {a};
  \node[main node] (2) at ({360/8 * (0) + 360/2 + 360/16}:\radius) {b};
  \node[main node] (3) at ({360/8 * (2) + 360/2 + 360/16}:\radius) {c};
  
  \node[main node] (7) at ({360/8 * (6) + 360/2 + 360/16}:\radius) {h};
  \node[main node] (8) at ({360/8 * (7) + 360/2 + 360/16}:\radius) {i};

  \path
  (7) edge [->] node {} (G)
  (8) edge [->] node {} (G)
  (G) edge [->] node {} (1)
      edge [->] node {} (3)
  (2) edge [->] node {} (G)     
  (2) edge [->] node {} (1)
  (3) edge [->] node {} (1)
  (4) edge [->] node {} (3)
  (6) edge [->] node {} (5)
  (8) edge [->] node {} (7)
  (7) edge [->,bend right,dashed] node {} (3)
  (5) edge [->,dashed] node {} (3);
  
\end{tikzpicture}
 
 \vspace{6.75mm}

 \begin{tikzpicture}[auto, thin, scale=0.85,
  main node/.style={circle,inner sep=.75mm,fill=white!10,draw,font=\sffamily\tiny\bfseries},
  colored node/.style={circle,fill=gray!10,draw,font=\sffamily\tiny\bfseries},
  ]
 \def \n {9}
 \def \radius {1.5cm}
 \def \margin {1} 
  \foreach \s/\na in {1/a,2/c,3/g,4/b,5/d,6/e,7/f,8/h,9/i}
  {
   \node[main node] (\s) at (\intcalcMul{\s}{\margin},0) {\na};
  }
  \path
  (2) edge [->] node {} (1)
  (3) edge [->,bend right] node {} (1)
      edge [->] node {} (2)
  (4) edge [->,bend right] node {} (1)
      edge [->] node {} (3)
  (5) edge [->,bend left] node {} (3)
      edge [->,bend right] node {} (2)
  (6) edge [->,bend left] node {} (3)
  (7) edge [->,bend left] node {} (3)
      edge [->] node {} (6)
  (8) edge [->,bend left] node {} (3)
  (9) edge [->,bend left] node {} (3)
      edge [->] node {} (8)
  (8) edge [->,bend right,dashed] node {} (2)
  (6) edge [->,bend right,dashed] node {} (2);
  
\end{tikzpicture}  
 
\end{center} 
\end{multicols} 
\end{figure}
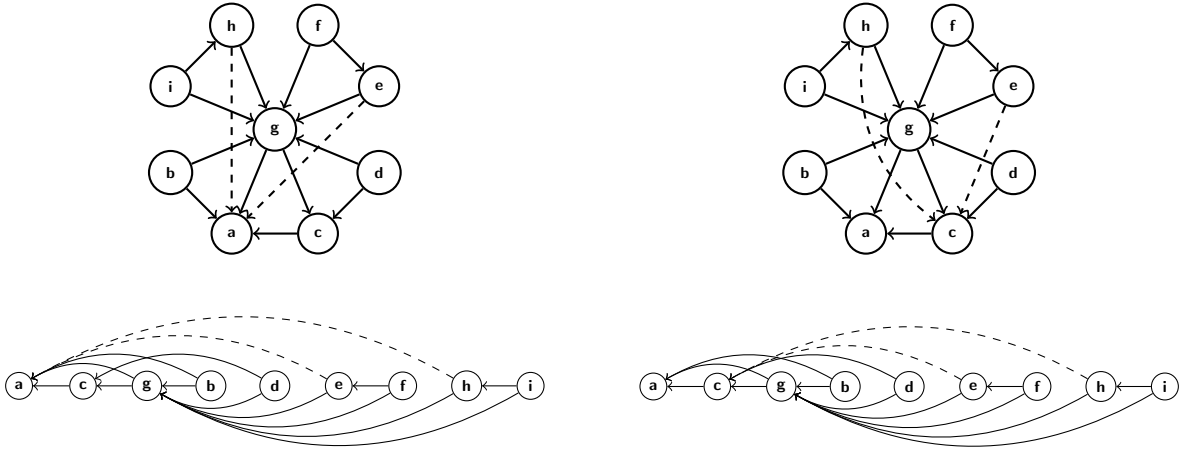

\newtheorem*{case}{Case: $\maj(h) = \maj(e)$}
\begin{case}
\label{separatemajorparents}
If $a<c<g$ and $h$ and $e$ both share the same major parent, then $g$ will be
colored before all the vertices of $N_g \setminus \lbrace{ g \rbrace }$ have
been activated.
\end{case}

\begin{proof}

Since $\maj (h) = c$ or $\maj (h) = a$, let us first assume that $\maj(h) = c$
and $\maj(e) = c$. As $\maj(e)=\maj(h)=\maj(d)=c$, by the time two vertices in
$\{d,e,h\}$ are activated $c$ will be colored. Now consider the last element to
be activated in the set $\lbrace { d, e, h \rbrace }$. That element must then
have $g$ as a mother, as $g$ is a parent to all elements in $\lbrace { d, e, h
\rbrace }$. Thus, when the last element is activated, Alice will act on $g$. As
a result, the activation of $i$, $f$, and the last element of $\{d,e,h \}$ to be
activated will all cause Alice to act on $g$, and $g$ will have been acted upon
twice and hence colored. Thus, $g$ will be colored before all of its surrounding
vertices have been activated. Note that the same argument holds when considering
$\maj(h)= a = \maj(e)$, by replacing $c$ with $a$ and $d$ with $b$.
\end{proof}

We have just shown that if $h$ and $e$ both share the same major parent, then as
long as Alice follows the activation strategy, $g$ will be colored before all of
its surrounding neighbors have been be activated. Now we must focus on the
remaining subcase, in which $h$ and $e$ have different major parents.

\newtheorem*{case2}{Case: $\maj(h) \neq \maj(e)$}
\begin{case2}
If $a<c<g$, and $h$ and $e$ do not share the same major parent, then $g$ will be
colored before all the vertices of $N_g \setminus \lbrace{ g \rbrace }$ have
been activated.
\end{case2}
As we have made no assumptions about the relative ordering of $e$ and $h$, we
can assume without loss of generality that $a = \maj (h)$ and $c = \maj (e)$.
The visual representation of this case can be see in Figure \ref{fig:figure4}.

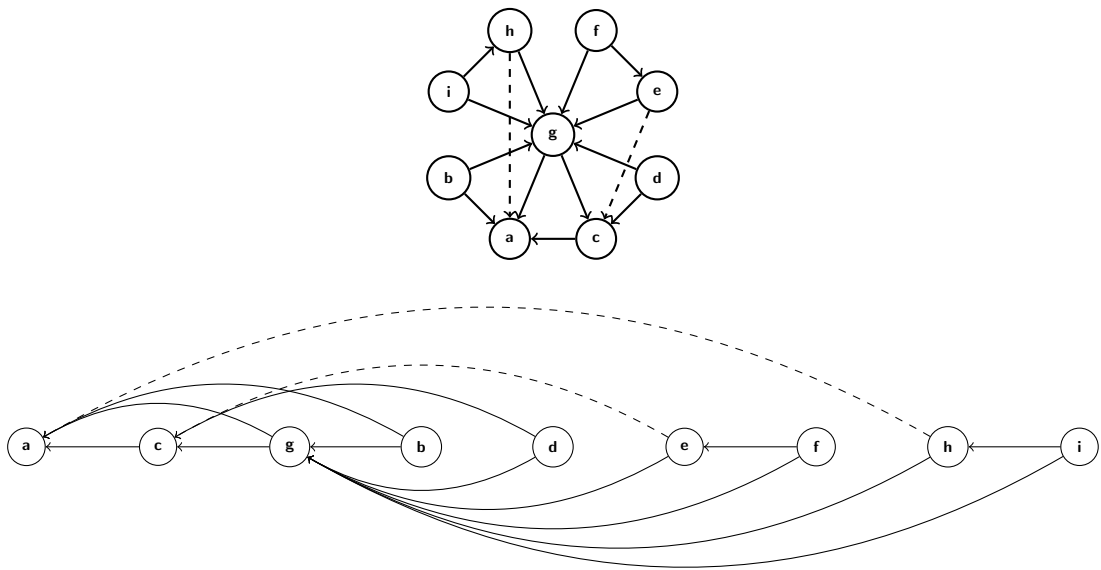
\begin{figure}[here]
\caption{$a<c<g$, $h$ and $e$ do not share a major parent.}
\label{fig:figure4}
\begin{center}
 \begin{tikzpicture}[auto, thick,
  main node/.style={circle,fill=white!10,draw,font=\sffamily\tiny\bfseries},
  colored node/.style={circle,fill=gray!10,draw,font=\sffamily\tiny\bfseries},
  ]

 \def \n {6}
 \def \radius {1.5cm}
 \def \margin {8} 
  \node[main node] (G) at (0,0) {g};
  \foreach \s in {4,...,\n}
  {
   \node[main node] (\s) at ({360/8 * (\s - 1) + 360/2 + 360/16}:\radius) {\makeAlph{\s}};
    \path
   (\s) edge [->] node {} (G);
  }
  \node[main node] (1) at ({360/8 * (1) + 360/2 + 360/16}:\radius) {a};
  \node[main node] (2) at ({360/8 * (0) + 360/2 + 360/16}:\radius) {b};
  \node[main node] (3) at ({360/8 * (2) + 360/2 + 360/16}:\radius) {c};
  
  \node[main node] (7) at ({360/8 * (6) + 360/2 + 360/16}:\radius) {h};
  \node[main node] (8) at ({360/8 * (7) + 360/2 + 360/16}:\radius) {i};

  \path
  (7) edge [->] node {} (G)
  (8) edge [->] node {} (G)
  (G) edge [->] node {} (1)
      edge [->] node {} (3)
  (2) edge [->] node {} (G)     
  (2) edge [->] node {} (1)
  (3) edge [->] node {} (1)
  (4) edge [->] node {} (3)
  (6) edge [->] node {} (5)
  (8) edge [->] node {} (7)
  (7) edge [->,dashed] node {} (1)
  (5) edge [->,dashed] node {} (3);
  
\end{tikzpicture}

\begin{tikzpicture}[auto, thin, scale=1.75,
  main node/.style={circle,inner sep=1.25mm,fill=white!10,draw,font=\sffamily\tiny\bfseries},
  colored node/.style={circle,fill=gray!10,draw,font=\sffamily\tiny\bfseries},
  ]
 \def \n {9}
 \def \radius {1.5cm}
 \def \margin {1} 
  \foreach \s/\na in {1/a,2/c,3/g,4/b,5/d,6/e,7/f,8/h,9/i}
  {
   \node[main node] (\s) at (\intcalcMul{\s}{\margin},0) {\na};
  }
  \path
  (2) edge [->] node {} (1)
  (3) edge [->,bend right] node {} (1)
      edge [->] node {} (2)
  (4) edge [->,bend right] node {} (1)
      edge [->] node {} (3)
  (5) edge [->,bend left] node {} (3)
      edge [->,bend right] node {} (2)
  (6) edge [->,bend left] node {} (3)
  (7) edge [->,bend left] node {} (3)
      edge [->] node {} (6)
  (8) edge [->,bend left] node {} (3)
  (9) edge [->,bend left] node {} (3)
      edge [->] node {} (8)
  (8) edge [->,dashed,bend right] node {} (1)      
  (6) edge [->,dashed,bend right] node {} (2);
\end{tikzpicture}  

\end{center} 
\end{figure}

\begin{claim}
When $a<c<g$ and $h$ and $e$ do not share the same major parent, if $a$ is
activated before $b$ is activated, then $g$ will be colored before all the
vertices in $N_g \setminus \{ g \}$ have been activated.
\end{claim}

\begin{proof}

Assume $a$ is activated before $b$ is activated. Considering that $a$ will be
acted upon in response to the activation of $h$, $i$, and $f$, if two or more of
these three vertices are activated before $b$ is activated, $a$ will have had at
least two actions taken on it by Alice before $b$ is activated. This means that
$a$ will have been colored before $b$ is activated, which runs contrary to our
condition that $b$ be activated before $a$ is colored. Hence, at most one
element of $\lbrace { h, i, f \rbrace }$ may be activated before $b$ (Claim
\ref{claim2}).

If $h$ is activated after $a$ is colored, then Alice will take action on $g$.
This ensures that Alice will win, since the activation of any vertices in $\{h,
i, f \}$ will cause Alice to act on $g$. So assume $h$ is activated before $a$
is colored. The only way this may occur is if $h$ is activated before either $i$
or $f$ are activated, as we have already shown that $b$ must be activated before
all elements of $\{i,f\}$ are activated. Thus, $h$ and $b$ are the first two
elements of $\lbrace {b, h, i, f \rbrace }$ to be activated, by which time $a$
will be colored. Hence once the third element of this set, be it $i$ or $f$, is
activated, both $g$ and $c$ will immediately be activated.

Now, considering that the major parent of $e$ is $c$, when $e$ is activated it
must be the case that $c$ is activated if uncolored, or $g$ is activated because
$c$ is already colored. Hence, after the first three elements of $\lbrace { b,
e, h, i, f \rbrace }$ are activated, we will have $a$ colored and $c$ active
(with $c$ possibly colored). Furthermore, when $d$ is activated, Alice will
either take action on $c$ if $c$ is uncolored or on $g$ if $c$ is colored. By
the time that the first four vertices of $\lbrace { b, e, d, h, i, f \rbrace }$
are activated, if $g$ has not been colored then we must have $a$ and $c$ both
colored. No matter what the fifth element of $\lbrace { b, e, d, h, i, f \rbrace
}$ is to be activated, its activation must in turn cause Alice to take action on
$g$. But since $g$ will then be its own mother, Alice will immediately color $g$
on that very same turn. Thus, Alice will have colored $g$ before all eight
surrounding vertices have been activated.

\end{proof}

\begin{claim}
When $a<c<g$ and $h$ and $e$ do not share the same major parent, if $a$ is
activated after $b$ is activated, then $g$ will be colored before all vertices
in $N_g \setminus \lbrace{ g \rbrace }$ have been activated.
\end{claim}
 
\begin{proof}
Since $b$ is activated before $a$, none of the elements of $\lbrace{ h, i, f
\rbrace}$ may be activated before $b$, as the activation of any of these
vertices would in turn cause Alice to act on $a$. However, if $b$ is the first
to be activated, Alice will then immediately activate $a$. Hence, when the first
two elements of $\lbrace{ b, h, i, f \rbrace}$ are activated, $a$ will
immediately be colored. By the time the first three elements of $\lbrace{ b, h,
i, f \rbrace}$ are activated, then $a$ will be colored, and $g$ will be active,
as well as $c$ if $g$ remains uncolored.

Since $e$ and $a$ are not neighbors, when $e$ is activated Alice will either
take action on $c$ if uncolored, or $g$ if $c$ is uncolored. Furthermore, when
$d$ is activated, $c$ will then be immediately activated if it was not already.
Now consider the first three elements to be activated in the set $\lbrace{ h, i,
f , d, e \rbrace}$. At least one of these elements must belong in the subset
$\lbrace{h, i, f \rbrace}$, but $b$ must be activated before any element of this
subset is activated.

As a result, by the time the first three elements of $\lbrace{ h, i, f , d, e
\rbrace}$ are activated, $b$ will have already been activated (as a result of
our hypothesis) and we will have four actions taken on $a$ and $c$. This means
that $a$ and $c$ will be colored before the fourth element of $\lbrace{ h, i, f
, d, e \rbrace}$ is activated. The activation of the fourth element of this set
will cause Alice to take action on $g$, at which point $g$ will be its own
mother. At this time Alice will color $g$, while the fifth element of $\lbrace{
h, i, f , d, e \rbrace}$ remains inactive.
\end{proof}

Thus, we have shown that no matter the case, so long as Alice follows the
activation strategy, Alice will always color a vertex before it no longer has a
legal color.
\phantom\qedhere
\end{proof}

\begin{corollary}
\label{chordalfourboundcorollary}
If Alice is playing the 2-clique-relaxed game on a partial $2$-tree $G$, then
Alice can always win with  $4$ colors.
\end{corollary}

\begin{proof}
Assume Alice and Bob are playing the coloring game on $G$. Since $G$ is a
partial $2$-tree, it is the subgraph of some chordal graph $H$ with clique
number $\omega(H) = 3$ and with the same vertex set as $G$. Alice will play the
$2$-clique-relaxed coloring game on $G$ as though she were playing the game on
$H$ using the Activation Strategy. As before, Alice's moves on $H$ will still be
valid in $G$. It suffices to show that for any uncolored vertex $v \in V(H)$,
the strategy provides a legal coloring for Alice, which is shown in Theorem
\ref{chordalfourbound} for 4 colors.
\end{proof}

Recall that every outerplanar graph is a partial $2$-tree. The above result is
hence a generalization of the result in \cite{clique-relaxed} where the authors
showed that for any outerplanar graph $G$, $\chi_\text{g}^{(2)}(G) \leq 4$. Now
we parametrize the clique-size of a chordal graph and the number of colors used
in the $k$-clique-relaxed game in order to provide a more general result
applicable to more forms of gameplay.

\begin{theorem}
\label{ubercooltheorem}
For any triplet of integers $c$, $k$ and $\omega$ such that $ck - 3 \omega + 1 >
0$ and $\omega > k$, if Alice and Bob play the $k$-clique-relaxed coloring game
on a chordal graph $H$ with clique number $\omega(H)=\omega$, then Alice will
win using the Activation Strategy with $c$ colors. That is, \[
\chi_\text{g}^{(k)}(H) \leq \Bigg{\lfloor} \frac{3 \omega - 1}{k} \Bigg\rfloor +
1.\]
\end{theorem}

\begin{proof}

Let $c,k$ and $\omega$ be integers such that $ck - 3 \omega + 1 > 0$ and $\omega
> k$, and let $H$ be a chordal graph with clique number $\omega(H) = \omega$.
Assume Alice and Bob are playing the $k$-clique-relaxed coloring game with $c$
colors, and Alice has chosen some simplicial linear ordering $\mathcal{L}$ on
the vertices of $H$. The only way an uncolored vertex $g$ can have no legal
coloring is if it is the sole common vertex to $c$ otherwise disjoint cliques
$\mathcal{C}_1 , \ldots, \mathcal{C}_{c}$ of size $k+1$, where for each
$\alpha_i \in [c]$ all vertices belonging to $\mathcal{C}_i$, with the exception
of $g$, are colored color $\alpha_i$. For each $\alpha_i \in [c]$, we will let
$V_i$ denote the set of vertices belonging to $\mathcal{C}_i$. We define $V_g =
\bigcup_{i=1}^{c} V_i$.

Let $\{ a_1 , \ldots a_j \}$ be the set of back-neighbors of $g$. By Lemma
\ref{backneighbors}, $g$ has at most $\omega -1 $ back-neighbors and so $j \leq
\omega - 1$. Note that $N^{+} [g] = \{ g, a_1 , \ldots, a_j \}$.

Let $X = V_g \setminus N^{+} [g]$. Consider any vertex $v \in X$. Since $g$ is
adjacent to $v$, and $v$ cannot be a parent of $g$, then it must be the case
that $g < v$, i.e. $g$ is a back-neighbor of $v$. Hence $g$ is a possible
candidate for $\maj(v)$. The only way that $\maj(v) \neq g$ is if $\maj(v) < g$.
But by Theorem \ref{simplicial} and the choice of a simplicial ordering, that
would imply that $g \adj \maj (v)$. Hence for all $v \in X$, $\maj (v) \in N^{+}
[g]$. By the same reasoning, any parent of $v$ that is less than $g$ must be
contained in $N^{+} [g]$. Thus, so long as $g$ remains uncolored, when any $v
\in X$ is activated, Alice will take action in $N^{+} [g]$.

Consider the size of $X$ as compared to the size of $N^{+} [g]$. On one hand,
$|X| = c (k+1 -1) - j \geq ck - \omega +1$; on the other, we have $|N^{+} [g]| =
j+1 \leq \omega$. However,
$ck - 3 \omega + 1 > 0 \mbox{ if and only if } ck - \omega +1> 2 \omega$.
Thus, $|X| \geq ck - \omega + 1 > 2 \omega \geq 2 |N^{+} [g]|$, or $|X| > 2 |N^{+} [g]|$.


Therefore, by the time all vertices in $X$ are activated, there will have been a
total of over $2 |N^{+} [g]|$ actions performed on all of the vertices in $N^{+}
[g]$. Since each vertex can be activated at most twice before being colored,
then before all the vertices in $X$ have been activated, all the vertices in
$N^{+} [g]$ will have been colored. Consequently $g$ will have been colored
before all of its surrounding vertices have been activated.
\end{proof}

\begin{corollary}
\label{ubercoolcorollary}
If Alice is playing the $k$-clique-relaxed game on a partial $\lambda$-tree $G$,
then Alice can always win with  $\lfloor \frac{3 \lambda + 2}{k} \rfloor + 1$
colors.
\end{corollary}

\begin{proof}

Assume Alice and Bob are playing the coloring game on $G$. Since $G$ is a
partial $\lambda$-tree, it is the subgraph of some chordal graph $H$ with clique
number $\omega(H) = \lambda+1$ and the same vertex set as $G$. Alice will play
the $k$-clique-relaxed coloring game on $G$ as though she were playing the game
on $H$ using the activation strategy. As before, Alice's moves on $H$ will still
be valid in $G$. It suffices to show that for any uncolored vertex $v \in V(H)$,
the strategy provides a legal coloring for Alice, which is shown in Theorem
\ref{ubercooltheorem} for $\lfloor \frac{3 \omega - 1}{k} \rfloor + 1$ colors.
Since $\omega = \lambda +1$, Alice only needs at most $\lfloor \frac{3 \lambda +
2}{k} \rfloor + 1$ colors to win.
\end{proof}

\begin{corollary}
If Alice is playing the k-clique-relaxed game on a partial $k$-tree $H$, then
Alice can always win with at most 4 colors.
\end{corollary}

\begin{proof}
Let $H$ be such a graph.

\begin{itemize}
\item{When $k \geq 3$, then by Corollary \ref{ubercoolcorollary}, 
$\chi_\text{g}^{(k)}(H) \leq \lfloor \frac{3 k + 2}{k}  \rfloor + 1 = \lfloor 3
+ \frac{2}{k} \rfloor + 1 = 3 + 1 = 4$.}

\item{When $k=2$, then by Corollary \ref{chordalfourboundcorollary}, Alice can win using 4 colors.}

\item{When $k=1$, then by Corollary \ref{k+3corollary}, Alice can again win using 4 colors.}

\end{itemize}
Thus, for any partial $k$-tree, Alice can win the $k$-clique-relaxed game using only 4 colors.
\end{proof}

This theorem generalizes to all chordal graphs the result in
\cite{clique-relaxed} where the authors showed that for any outerplanar graph
$G$, $\chi_\text{g}^{(2)}(G) \leq 4$.

\section{Thoughts}

We are of the opinion that our bounds on $\chi_\text{g}^{(k)}$ for chordal
graphs can be improved upon, as the strategies used by Alice are geared towards
the original and relaxed coloring game. Indeed, one immediate advantage of the
clique-relaxed game is that the only vertices Alice need worry about in a graph
are those which belong to a clique of size $k+1$ or greater. When playing the
$k$-clique-relaxed coloring game, the vertices not belonging to a clique of size
$k+1$ or greater can all be legally colored with an arbitrary color. However,
thus far we have not come up with a strategy which takes advantage of this fact
and reduces the bound on the clique-relaxed chromatic number. In addition, one
will note that in our strategies, we have not specified what color Alice ought
to label a vertex with: we believe that a higher precision in color choice could
tighten the bounds on the clique-relaxed game chromatic number, but again thus
far we have not found a way to incorporate strategic color choices into a
strategy.

Moreover, we believe that strategies geared towards chordal graphs can be
modified for use on a new, broad class of graphs, the class of ($a$,$b$)-pseudo
chordal graphs first presented in \cite{pseudo}. These graphs possess as a main
underlying structure a chordal graph; the parameters $a$ and $b$ serve as a
measurement of how far a graph is from being chordal. However, in order to
extend our methods to this new class of graphs, a new algorithm further geared
towards cliques instead of just regular coloring or defect coloring is required.

Furthermore, similar to the previous Linfield REU group working on outerplanar
graphs, who believed that the $2$-clique-relaxed game chromatic number of any
outerplanar graph is at most $3$, we believe that the $2$-clique-relaxed game
chromatic number of any chordal graph with clique number $3$ is at most $3$.
However, so far we have been unsuccessful in proving this bound.


\begin{thebibliography}{10}

\addtolength{\leftmargin}{0.2in} 
\setlength{\itemindent}{-0.2in}

    \bibitem{West} Douglas B. West \emph{Introduction to Graph Theory}, 2nd edition, p.224-226, 2001

    \bibitem{clique-relaxed} Charles Dunn, Jennifer Firkins Nordstrom, Cassandra Naymie, Erin Pitney, William Sehorn and Charlie Suer, ``Clique-relaxed Graph Coloring," \emph{Involve, a journal of Mathematics}, Vol.4, no.2, 2011

    \bibitem{chordal} Charles Dunn, H.A. Kierstead ``A simple competitive graph coloring algorithm II," \emph{Journal of Combinatorial Theory}, Series B 90, p.93-106, 2004
    
    \bibitem{pseudo} X. Zhu ``The game coloring number of pseudo partial $k$-trees," \emph{Discrete Math} 215, p.245-262, 2000
    
    \bibitem{defective} C. Chou, W. Wang and X. Zhu ``Relaxed game chromatic number of graphs" \emph{Discrete Mathematics} 262, p.89-98, 2003
    
    \bibitem{forests} U. Faigle, U. Kern, H. Kierstead and W.T. Trotter  ``On the Game Chromatic Number of Some Classes of Graphs" \emph{Ars Combinatoria} 35, p.143-150, 1993
    
\end{thebibliography}
\end{document}